\documentclass[12pt]{amsart}
\usepackage{amsmath,amssymb,amsfonts,a4wide}

\def\be{\begin{equation}}
\def\ee{\end{equation}}
\def\lbl{\label}

\newtheorem {thm}{Theorem}
\newtheorem {proposition}[thm]{Proposition}

\begin{document}

\title{On tails of perpetuities}
\author[Pawe{\l} Hitczenko]{Pawe{\l} Hitczenko$^\dag{}$}
\address{Pawe{\l} Hitczenko\\
Departments of Mathematics and Computer Science  \\
Drexel University\\
Philadelphia, PA 19104 \\
U.S.A}
\email{phitczenko@math.drexel.edu}
\urladdr{
http://www.math.drexel.edu/$\sim$phitczen}

\thanks{$^\dag{}$ Supported in part by the NSA grant  \#H98230-09-1-0062}

\title{On tails of perpetuities} 

\begin{abstract}
We establish an upper bound on the tails of a random variable that arises as a solution of a stochastic difference equation. In the non--negative case our bound is similar  to a lower bound obtained by Goldie and Gr\"ubel in 1996. 
\end{abstract}

\vspace{1cm}



\maketitle

\section{Introduction} 

A random variable $R$ satisfying the distributional identity
\be\lbl{eq:r}R\stackrel d=MR+Q,\ee
where $(M,Q)$ are independent of $R$ on the right-hand side and
$\stackrel d=$ denotes the equality in distribution, is referred to as
perpetuity and plays an important role in applied probability. The
main reason for this is that  it
appears as a limit in distribution of a sequence $(R_n)$ given by 
\[R_n\stackrel d= M_nR_{n-1}+Q_n,\quad n\ge1,\]
provided that limit exists (here, $(M_n,Q_n)$ is a sequence of i.i.d.  random vectors
distributed like $(M,Q)$ and $R_0$ could be an arbitrary random variable; for convenience we will set $R_0=0$).
Systematic study of properties of such sequences  was initiated by
Kesten in \cite{kesten} and they continue till this day. Once the
convergence in distribution of $(R_n)$ is established, at the center
of the investigation is the tail behavior of $R$. There are two
distinctly different cases:
\[P(|M|>1)>0\quad\mbox{and}\quad P(|M|\le1)=1.\]
The first results in $R$ having  a heavy tail distribution, that is 
\[P(|R|>x)\sim Cx^{-\kappa},\]
for a suitably chosen constant $\kappa$ and some constant $C$  (see the
original paper of   Kesten \cite{kesten} or \cite{goldie}), while in
the second case the tails of $R$ are no heavier than exponential. This
was observed by Goldie and Gr\"ubel in \cite{gg}. Some subsequent work
is in \cite{hw}, but the full picture in this case is not
complete. The purpose of this note is to shed some additional light
on that case by establishing a universal upper bound on the tails of $|R|$. In a special, but important, situation when  $Q$ and  $M$ (and thus also $R$) are non--negative our bound is comparable to a lower bound obtained by Goldie and Gr\"ubel in \cite{gg}.

\section{Bounds on the tails}
For a random variable $M$ such that $|M|\le 1$ and $0<\delta<1$ define  $p_\delta:=P(1-\delta\le |M|\le1)$. Then, as has been shown in \cite{gg} (see also the equation (2.2) in
\cite{hw}) if $0\le M\le 1$ and $Q\equiv q$ ($q$ being a positive
constant), then for $0<c<1$ and $x>q$ we have
\[P(R> x)\ge \exp(\frac{\ln(1-c)}{\ln(1-cq/x)}\ln p_{cq/x}).\]
Since $\ln(1-cq/x)\le -cq/x$,  for any particular value of $c$, say  $c=1/2$,  this immediately gives
\[P(R> x)\ge \exp(-\frac{\ln(1-c)}{cq}
x\ln(p_{cq/x}))=\exp(\frac{2\ln2}q x\ln p_{q/(2x)}).\]
Our aim here is to supply an upper bound of a similar form. 
 While our result does not give the asymptotics of $P(R>x)$
as $x\to\infty$, it shows that it essentially behaves like $\exp(\frac{c_1}q
x\ln p_{c_2q/x})$ for some positive constants $c_1,c_2$.
Specifically, we prove 

\begin{proposition}Assume $|Q|\le q$, $|M|\le1$, and let $R$ be given by
\eqref{eq:r}. Then, for sufficiently large $x$
\[P(|R|> x)\le\exp(\frac 1{4q}x\ln p_{2q/x}).  
\]
Thus, if $Q\equiv q>0$ and $0\le M\le 1$ then 
\[\exp(\frac{2\ln2}q x\ln p_{q/(2x)}) \le P(R> x)\le\exp(\frac1{4q} x\ln p_{2q/x}).  
\]
\end{proposition}

\begin{proof} If $P(|M|=1)>0$ then,  as was proved in \cite{gg}, $R$ has tails bounded by those of an
exponential variable, so we assume  that $|M|$ has no atom at 1. Fix $0<\delta<1$  and 
define a sequence  $(T_k)$ as follows
\[T_0=0,\quad T_m=\inf\{k\ge1:\ |M_{T_{m-1}+k}|\le1-\delta\},\quad
m\ge1.\]
Then $T_k$'s are i.i.d. random variables, each having a
geometric distibution with parameter $1-p_\delta$. Furthermore,
$|M_k|\le1-\delta$ if $k=T_1+\dots+T_i$ for some $i\ge1$ and $|M_k|\le1$
otherwise. Therefore, 
\[\prod_{k=1}^m|M_k|\le (1-\delta)^{j}\quad{\rm for}\quad
T_1+\dots+T_j\le m<T_1+\dots+T_j+T_{j+1}.\]
This in turn implies that 
\[\Big|\sum_{k\ge1}\prod_{j=1}^{k-1}M_j\Big|\le
\sum_{k\ge1}\prod_{j=1}^{k-1}|M_j|\le
T_1+(1-\delta)T_2+(1-\delta)^2T_3+\cdots=\sum_{k\ge1}(1-\delta)^{k-1}T_k.\]
Therefore, if $|Q|\le q$  we get
\be\lbl{eq:up_bdd}P(|R|>x)\le P(\sum_{k\ge1}\prod_{j=1}^{k-1}|M_j|\ge\frac xq)\le
P(\sum_{k\ge1}T_k(1-\delta)^{k-1}\ge \frac xq).\ee
To bound the latter probability we use a widely known argument (our
calculations follow
\cite[proof of Proposition~2]{gh}). First, if $T$ is a geometric variable with parameter
$1-p$ then 
\[Ee^{\lambda T}=\sum_{j=1}^\infty e^{\lambda j}P(T=j)=
\sum_{j=1}^\infty e^{\lambda
  j}p^{j-1}(1-p)=\frac{e^\lambda(1-p)}{1-e^\lambda
  p}=\frac{e^\lambda}{1-\frac p{1-p}(e^\lambda-1)},\]
provided $e^\lambda p<1$. Thus, writing $t$ in place of $x/q$ in the right-hand side of
\eqref{eq:up_bdd}, for $\lambda>0$ we have
\[P(\sum_{k\ge1}(1-\delta)^{k-1}T_k\ge
t)=P(\exp(\lambda\sum_{k\ge1}(1-\delta)^{k-1}T_k)\ge e^{\lambda t})\le
e^{-\lambda t}Ee^{\lambda\sum_{k\ge1}T_k(1-\delta)^{k-1}}.\]
If $\lambda$ satisfies $e^\lambda p<1$
then $pe^{\lambda(1-\delta)^{k-1}}<1$ for every $k\ge 1$ as well, and by independence of $(T_k)$,  the expectation on the right is 
\be\lbl{eq:prod}\prod_{k=1}^\infty\frac{e^{\lambda(1-\delta)^{k-1}}}{1-\frac p{1-p}(e^{\lambda(1-\delta)^{k-1}}-1)}=e^{\lambda/\delta}\prod_{k=1}^\infty\frac1 {1-\frac p{1-p}(e^{\lambda(1-\delta)^{k-1}}-1)}.\ee
Now, choose $\lambda>0$ so that $\frac p{1-p}(e^\lambda-1)\le
\frac12$. Then, as $1/(1-u)\le e^{2u}$ for $0\le u\le 1/2$, for every $k\ge1$ we get 
\[\frac1{1-\frac p{1-p}(e^{\lambda(1-\delta)^{k-1}}-1)}\le \exp(2\frac
p{1-p}(e^{\lambda(1-\delta)^{k-1}}-1))
.\] 
Therefore, the rightmost product in \eqref{eq:prod} is bounded by 
\[\exp(2\frac
p{1-p}\sum_{k\ge1}(e^{\lambda(1-\delta)^{k-1})}-1)).
\]
We bound the sum in the exponent  as follows 
\begin{eqnarray*}&&\sum_{k\ge1}\sum_{j\ge1}\frac{\lambda^j(1-\delta)^{(k-1)j}}{j!}
=\sum_{j\ge1}\frac{\lambda^j}{j!}\sum_{k\ge1}(1-\delta)^{j(k-1)}
\\&&\qquad\quad
=\sum_{j\ge1}\frac{\lambda^j}{j!}\frac1{1-(1-\delta)^j}
\le\frac1\delta\sum_{j\ge1}\frac{\lambda^j}{j!}
=\frac{e^\lambda-1}\delta.
\end{eqnarray*}
Combining the above estimates we get that 
\be\lbl{eq:up_bdd2}P(|R|>qt)\le
\exp(-t\lambda+\frac\lambda\delta+\frac{2p}{1-p}\frac{e^\lambda-1}\delta).
\ee
provided that $\lambda$ satisfies the required conditions, that is:
\[e^\lambda p<1\quad \mbox{and}\quad \frac
p{1-p}(e^\lambda-1)\le\frac12.\] 
Clearly both are satisfied when $e^\lambda p\le1/2$. 

We finish the proof by  making a
suitable choice of $\lambda$.
Since we are
assuming that $|M|$ has no atom at 1 and we are interested in large $x$, we
may assume that $\delta$ is small enough so that $p_\delta<1/3$.
This condition 
implies that $2p_\delta/(1-p_\delta)<3p_\delta$  so
that the last term in the exponent of \eqref{eq:up_bdd2} is bounded by
$3p_\delta(e^\lambda-1)/\delta$. 
Now let 
$t=2/\delta$. Then
\eqref{eq:up_bdd2} becomes
\[P(|R|>qt)\le\exp(-\lambda\frac 2\delta+\frac\lambda\delta+\frac{2p_\delta}{1-p_\delta}\frac{e^\lambda-1}\delta)\le\exp(-\frac1\delta(\lambda-3p_\delta(e^\lambda-1))).
\]
Set  $\lambda=\ln(\frac1{3p_\delta})$ so that $e^\lambda
p_\delta=\frac13$.  
This choice of $\lambda$ is within the constraints and maximizes the
value of
$\lambda-3p_\delta(e^\lambda-1)$, this maximal value being 
\[\ln(\frac1{3p_\delta})-3p_\delta(\frac1{3p_\delta}-1)=\ln(\frac1{p_\delta})-(1+\ln3)+3p_\delta\ge\frac12\ln(1/p_\delta),
\]  
with the inequality valid for sufficiently small $p_\delta$ (less than $e^{-2}/9$ for
example). Thus, using $t=2/\delta$ we finally obtain
\[P(|R|>qt)\le\exp(-\frac1{2\delta}\ln(1/p_\delta))=\exp(\frac t4\ln p_{2/t}),
\]
or, in terms of $x$,
\[P(|R|>x)\le \exp(\frac x{4q}\ln p_{2q/x}).
\]
\end{proof}


\begin{thebibliography}{99}
\footnotesize

\bibitem{gh}
{\sc Goh,  W.~M.~Y. and Hitczenko, P.} (2008). 
 Random partitions with restricted part sizes.
{\em Random Structures Algorithms}. {\bf 32,} 440--462.

\bibitem{goldie}
{\sc  Goldie, C.~M.} (1991). 
Implicit renewal theory and tails of solutions of random equations.
{\em Ann. Appl. Probab.} {\bf 1,} 126--166.

\bibitem{gg}
{\sc  Goldie, C.~M. and Gr{\"u}bel, R.} (1996).
 Perpetuities with thin tails.
 {\em Adv. in Appl. Probab.} {\bf 28,} 463--480.

\bibitem{hw}
{\sc Hitczenko, P.  and Weso{\l}owski, J.} (2009).
Perpetuities with thin tails, revisited.
{\em Ann. Appl. Probab.} {\bf 19,} 2080 -- 2101.
\newblock (Corrected version available at http://arxiv.org/abs/0912.1694.)

\bibitem{kesten}
{\sc Kesten, H.} (1973).
Random difference equations and renewal theory for products of random
  matrices.
{\em Acta Math.} {\bf 131,} 207--248.

\end{thebibliography}
\end{document}